\documentclass[12pt]{amsart}
\usepackage{enumerate}
\usepackage{amsthm}
\usepackage{amsmath,amssymb,latexsym,amscd}
\usepackage{tikz}
\makeatletter
\@namedef{subjclassname@2020}{%
  \textup{2020} Mathematics Subject Classification}
\makeatother
\theoremstyle{definition}
\newtheorem{theorem}{Theorem}[section]
\newtheorem{lemma}[theorem]{Lemma}
\newtheorem{proposition}[theorem]{Proposition}
\newtheorem{corollary}[theorem]{Corollary}
\theoremstyle{definition}
\theoremstyle{definition}
\newtheorem{definition}[theorem]{Definition}

\newtheorem{example}[theorem]{Example}
\usepackage{indentfirst}

\begin{document}
\baselineskip=17pt
\title[]{On $\theta$-Hurewicz and $\alpha$-Hurewicz Topological spaces}
\author[Gaurav Kumar${^1}$, Sumit Mittal$^{2}$, Brij K. Tyagi]{ Gaurav Kumar${^1}$, Sumit Mittal$^{2}$, Brij K. Tyagi}
\address{ Gaurav Kumar \newline 
	\indent Department of Mathematics\indent \newline\indent University of Delhi\newline\indent 
	Delhi -- 110007, India.}
\email{gauravkumar.maths@dsc.du.ac}
\address{ Sumit Mittal \newline\indent 
Department of Mathematics\newline\indent University of Delhi\newline\indent 
Delhi -- 110007, India.}
\email{sumitmittal1105@maths.du.ac.in}
\address{Brij K. Tyagi\newline\indent Atma Ram Sanatan Dharma College\newline\indent University of Delhi -- 110021, India.}
\email{brijkishore.tyagi@gmail.com}

\date{}
\thanks{}
\begin{abstract} 
\noindent In this paper, we introduced  $\alpha$-Hurewicz $\&$  $\theta$-Hurewicz   properties  in a topological space $X$  and investigated their relationship with other selective covering properties. We have shown that for an extremally disconnected semi-regular spaces, the properties:  Hurewicz, semi-Hurewicz, $\alpha$-Hurewicz, $\theta$-Hurewicz, almost-Hurewicz, nearly Hurewicz and midly Hurewicz are equivalent. We have also proved that for an extremally disconnected space X, every finite power of X has $\theta$-Hurewicz property if and only if X has the selection principle $U_{fin}(\theta$-$\Omega, \theta$-$\Omega)$. The preservation  under several types of mappings  of $\alpha$-Hurewicz and $\theta$-Hurewicz properties are also discussed. Also, we showed that, if $X$ is a mildly Hurewicz subspace of $ \omega^\omega$, than $X$ is  bounded.

\end{abstract}


\maketitle
\section {Introduction and preliminaries}

The  classical Hurewicz property has a  long history from the paper \cite{H}. A topological space $X$ has Hurewicz property if for each sequence $(\mathcal{A}_k : {k \in\mathbb{N}})$ of open covers of $X$ there exists a sequence $(\mathcal{B}_k: {k \in\mathbb{N}})$ where   for each $k,$ $\mathcal{B}_k$ is a ﬁnite subset of $\mathcal{A}_k$ such that for each $x\in X$, $x\in \bigcup \mathcal{B}_k$ for all but finitely many $k.$ The Hurewicz property is weaker than the $\sigma$-compactness and stronger than  the Lindel$\ddot{o}$f property.\footnote{The    author$^1$ acknowledges the fellowship grant of University Grant Commission, India.
	
	 \indent  The    author$^2$ acknowledges the fellowship grant of Council of Scientific $\&$ Industrial Research, India.}

Recently, several weak variants of Hurewicz property have been studied after applying the interior
and the closure operators in the definition of a Hurewicz Property. Also, the other ways
have been examined when the sequence of open covers are replaced with generalized
open sets. For the study of the  variants of Hurewicz spaces, the readers can see \cite{AZ, XYZ, KSDS, ZZ, G, BN, MM}.

In this paper, we examine  the covering properties namely: $\alpha$-Hurewicz, $\theta$-Hurewicz which are alike to the classical Hurewicz property. 

The following generalizations of open sets will be used for definitions of variations on the Hurewicz property:

A subset $A$ of a topological   space $X$ is said to be:
\begin{itemize}
	\item $\theta$-open  \cite{B} if for every $x$ in  $A,$ there exists an open set $B$ in $X$ such that $x \in B \subset Cl({B}) \subset A$.
	\item $\alpha$-open \cite{ND} if $A\subset Int(Cl(Int(A))),$  or equivalently, if there exists an open set $B$ in $X,$ such that $B \subset A \subset Int(Cl(B))$. The complement of $\alpha$-open set is $\alpha$-closed. Moreover, a set $A$ is $\alpha$-closed in $X$ if $Cl(Int(Cl(A))) \subseteq A$. 
	\item semi-open \cite{LN} if there exists an open set $B$ in $X$ such that $B \subset A \subset Cl(B),$ or equivalently, if $A \subset Cl(Int(A))$. The set   $SO(X)$ denotes the collection  of all semi-open sets. The complement of a semi-open set is semi-closed, $sCl(A)$ denotes the semi-closure of  $A,$ $sCl(A)$  is intersection of all semi-closed sets containing $A$. $A\subseteq X$ is semi-closed if and only if $sCl(A) = A$.
\end{itemize}
Clearly, we have:
\begin{center}
	clopen $\Rightarrow$ $\theta$-open $\Rightarrow$ open $\Rightarrow \alpha$-open  $\Rightarrow$ semi-open.\end{center}

A space $X$ is said to be  $\alpha$-compact \cite{M} ($\theta$-compact \cite{KD}) if every $\alpha$-open ($\theta$-open) cover of $X$ has a finite subcover.\\ 
\indent  Recall that a space $X$ is said to be semi-regular \cite{T} if for each $x\in X$ and  for each semi-closed set $A$ such that $x\not\in A$, there exist disjoint
semi-open sets $B$ and $C$ of $X$ such that $x \in B $ and $A \subset C $.
\begin{lemma} \label{PP}\cite{T} For a space $X$ the following statements are equivalent:
	\item(i) $X$ is semi-regular;
	\item(ii) For each $x \in X $ and $A\in SO(X)$ such that $x \in A$, there exists a  $B \in SO(X)$ such that $x \in B \subset sCl({B}) \subset A$.
\end{lemma}

\begin{definition}
	\textnormal{A space $X$ is called semi-Hurewicz \cite{KSDS}   ( resp.,  mildly Hurewicz \cite{MM}, $\theta$-Hurewicz, $\alpha$-Hurewicz )  if for each sequence $( \mathcal{A}_k : {k \in\mathbb{N}})$  of semi-open ( resp., clopen, $\theta$-open, $\alpha$-open ) covers of $X$, there exists a sequence $(\mathcal{B}_k : {k \in\mathbb{N}})$ such that  for each $k\in\mathbb{N},$  $\mathcal{B}_k$ is a finite subset of $\mathcal{A}_k$ and for each $x\in X$, $x\in\bigcup \mathcal{B}_k$ for all but finitely many $k.$}
\end{definition}
Evidently, we have the following implications:
\begin{center}
	semi-Hurewicz $\Rightarrow$ $\alpha$-Hurewicz  $\Rightarrow$ Hurewicz $\Rightarrow$ $\theta$-Hurewicz $\Rightarrow$ Mildly Hurewicz.
\end{center}
The above properties are also  written in the form of  selection principle. Let
$\mathcal{A}$ and $\mathcal{B}$ be the collections  of subsets of a space $X$. Then  a space $X$ satisfies the selection principle :
$U_{fin}(\mathcal{A}, \mathcal{B})$ if for each sequence $(\mathcal{A}_k : k \in\mathbb{N})$ in $\mathcal{A}$ there exists a sequence $(\mathcal{B}_k : k \in \mathbb{N}),$ where for each $k,$ $\mathcal{B}_k$ is a finite subset of $\mathcal{A}_k$  such that  $\{\cup\mathcal{B}_k: k = 1,2,3,. . .\}$  is in $\mathcal{B}$  \cite{X}.
An infinite cover $\mathcal{C}$ of $X$ is said $\gamma$-cover (resp., $c$-$\gamma$-cover, $\theta$-$\gamma$-cover, $\alpha$-$\gamma$-cover $s$-$\gamma$-cover) if each element of $\mathcal{C}$ are open $(resp, clopen$, $\theta$-open, $\alpha$-open, semi-open) such that  for each $x\in X$, the set $\{U \in \mathcal{C} : x\not\in U\}$ is finite. Let $\Gamma$, $c$-$\Gamma$ $\theta$-$\Gamma$, $\alpha$-$\Gamma $ $s$-$\Gamma$ denotes the collection of all $\gamma$, $c$-$\gamma$, $\theta$-$\gamma$, $\alpha$-$\gamma$ $s$-$\gamma$ covers of $X,$ respectively and $\mathcal{O},$ $\mathcal{CO}$,   $\theta$-$\mathcal{O}$, $\alpha$-$\mathcal{O}$,  s-$\mathcal{O}$   denotes the collection of all open, clopen,    $\theta$-open, $\alpha$-open, semi-open covers of a space $X$, respectively. Then the
Hurewicz, mildly Hurewicz, $\theta$-Hurewicz, $\alpha$-Hurewicz, semi-Hurewicz property of $X$ is equivalent to selection principles: $U_{fin}(\mathcal{O}, \Gamma)$, $U_{fin}(\mathcal{CO}$, $c$-$\Gamma)$, $U_{fin}(\theta$-$\mathcal{O},$ $\theta$-$\Gamma)$, $U_{fin}(\alpha$-$\mathcal{O}$, $\alpha$-$\Gamma)$, $U_{fin}(s$-$\mathcal{O}$, $s$-$\Gamma)$, respectively. In this paper we study $\alpha$-Hurewicz and $\theta$-Hurewicz properties in details.

Throughout  the paper a space $X$ and $(X, \tau)$,  means a topological space $\&$ $|X|$ denotes the cardinality of $X$. For a subset    $A$ of a  space $X$, $Int(A)$ and   $\overline{A}$ or $ Cl({A})$,  denotes the  interior and closure of $A$ respectively. Further, $\omega$ and $\omega_1$ denote the first infinite cardinal $\&$ uncountable cardinal respectively.

\section{ The $\theta$-Hurewicz Spaces and $\alpha$-Hurewicz Spaces}   

First, recall that the family of all $\theta$-open (resp., $\alpha$-open) sets of a space $(X, \tau)$  are form topologies on $X,$ denoted by $\tau_\theta$  \cite{J} (resp., $\tau_\alpha$ \cite{ND}.
Further, $\tau_\theta \subseteq \tau\subseteq   \tau_\alpha$.  The role of $\theta$-open and $\alpha$-open sets have been invastigated in many papers (see, \cite{M, KD})

Clearly, a space $(X, \tau)$ is $\theta$-Hurewicz (resp., $\alpha$-Hurewicz) if and only if the space $(X, \tau_\theta)$ (resp., $(X, \tau_\alpha$) is Hurewicz.

\begin{theorem}\label{BKM}
	Every countable space $X$ has $\alpha$-Hurewicz property.
\end{theorem}
\begin{proof}
	Let $X=\{x_1, x_2, ...., x_n,....\}$ be a countable space. Let $<\mathcal{A}_k>_{k \in\mathbb{N}}$ be  a sequence of $\alpha$-open covers of $X. $ For each $k\in \mathbb{N},$ Consider $\mathcal{B}_k = \{A_{k,1}, A_{k,2},......A_{k,k}\},$ where for each $i\in \{1,2,....k\},$ $A_{k,i}$ is  $\alpha$-open such that $x_i\in A_{k,i}.$  Then $\mathcal{B}_k$ is a finite subset of $\mathcal{A}_k$ and for each $x\in X$, $x\in\bigcup \mathcal{B}_k$ for all but finitely many $k.$
\end{proof}
Similarly we can prove that every countable space has $\theta$-Hurewicz property.

\begin{example} \begin{enumerate}
		\item Every $\alpha$-compact space is $\alpha$-Hurewicz.  But the converse is not true. The real line $\mathbb{R}$ with the cocountable topology is   $\alpha$-Hurewicz being  semi-Hurewicz \cite{KSDS} but it is not $\alpha$-Compact since every $\alpha$-compact space is compact.
		\item Every $\theta$-compact space is $\theta$-Hurewicz.  But the converse is not true. Let $X$ be a countably infinite discrete space. Then the space $X$ has  $\theta$-Hurewicz property but the $\theta$-open cover $\{\{x\}: x\in X\}$ has no finite subcover.
		\item  The real line $\mathbb{R}$ is a Hurewicz space but it is not semi-Hurewicz \cite{KSDS}.
		\item The Sorgenfrey line $S$  does not have the  $\alpha$-Hurewicz property because it does not have  the Hurewicz property.
	\end{enumerate}
\end{example}

\begin{example} Let $A$ be a finite subset of an uncountable set $X.$ Then  $\tau = \{\phi,A, X\}$ is a  topology on $X.$ Clearly, the  space  $(X, \tau)$ is  Hurewicz. Moreover, the sets of the forms $A \cup\{p\},$ for  $p$ $\in X \setminus A$, are $\alpha$-open in $(X, \tau)$.  For each $k\in \mathbb{N}, $ put  $\mathcal{A}_k  = \{A \cup \{p\}: p\in X\setminus A \}.$  Then the sequence  $(\mathcal{A}_k : {k \in\mathbb{N}}) $  witnesses that $(X, \tau)$ is not an $\alpha$-Hurewicz space because  the cover $\mathcal{A}_k$ does not have a countable subcover.   
\end{example}
\begin{example}
	Let $p$ be a fixed point of an uncountable set $X$. Then $\tau_p = \{O \subseteq X : p \in O  \}$ together with the empty set is an uncountable particular point topology on $X$. It is shown in  \cite{L} that the space $X$ is not   Lindelöf, so $X$ can not be Hurewicz since every Hurewicz space is Lindel$\ddot{o}$f.  Note that the space $X$ is an only closed set containing $p.$   Then $Cl(A) = X$ for each $A\not=\emptyset$, $A\in \tau_p.$ Hence  $\phi$ and $X$ are  only $\theta$-open sets. Therefore $X$ is $\theta$-Hurewicz space.
\end{example}
A space $X$ is said to be nearly Hurewicz \cite{BN} (resp., almost Hurewicz \cite{G}) if for each sequence $(\mathcal{A}_k : k\in \mathbb{N})$ of open covers of $X$, there exists a sequence $(\mathcal{B}_k : k\in \mathbb{N})$, where for each $k\in \mathbb{N},$ $\mathcal{B}_k$ is a finite subset of $\mathcal{A}_k$ such that  for each $x\in X$, $x\in \cup\{Int(Cl(B)) : B\in \mathcal{B}_k\}$ (resp., $x\in \cup\{Cl(B) : B\in \mathcal{B}_k\})$ for all but finitely many $k$.

Evidently, from the definitions follows the following implications:
\begin{center} Hurewicz $\Rightarrow$  nearly Hurewicz $\Rightarrow$ almost Hurewicz.\end{center}

The following theorem describes a relation between almost Hurewicz and $\theta$-Hurewicz space
\begin{theorem}
	Every almost Hurewicz space is $\theta$-Hurewicz.
\end{theorem} 
\begin{proof}
	Let $X$ be an almost Hurewicz space and $(\mathcal{A}_k : k \in \mathbb{N})$ be a sequence of $\theta$-open covers of $X$. Then for each $k \in \mathbb{N}$ and each $x\in X$ there is an  open set $B_{x, k}$ such that $x\in B_{x,k}\subset Cl(B_{x,k})\subset A_k$ for some $A_k\in \mathcal{A}_k$.  For each $k$, put $\mathcal{B}_k = \{B_{x,k} : x\in X\}$. Then  each $\mathcal{B}_k$ is an open cover of $X$. Since $X$ is an almost Hurewicz, for each $k\in \mathbb{N},$ there is a finite subset $\mathcal{B}'_k$ of $\mathcal{B}_k$ such that $x\in X$, $x\in \cup\{Cl(B') : B' \in \mathcal{B}'_k\}$ for all but finitely many $k.$  Since for each $B'\in \mathcal{B}'_k$, there is a $A'_{k,B'}\in \mathcal{A}_k $  such that $Cl(B') \subset A'_{k, B'}.$ Let $\mathcal{A}'_k = \{A_{k, B'}\in \mathcal{A}_k : B'\in \mathcal{B}'_k\}$. Then  the sequence $(\mathcal{A}'_k : k\in \mathbb{N})$ witnesses  that $X$ is $\theta$-Hurewicz.
\end{proof}
Next we determine a class of spaces in which above  variants of  Hurewicz property are equivalent. Recall that, a space $X$ is called extremally disconnected if closure of open set is open. 

\begin{theorem}\label{XVB}
	For  an extremally disconnected semi-regular space $X$. The following statements are equivalent:
	\begin{enumerate}
		\item  $X$ is semi-Hurewicz;
		\item  $X$ is $\alpha$-Hurewicz;
		\item   $X$ is Hurewicz;
		\item  $X$ is nearly Hurewicz;
		\item   $X$ is almost Hurewicz;
		\item  $X$ is $\theta$-Hurewicz;
		\item  $X$ is  mildly Hurewicz.
	\end{enumerate}   
\end{theorem}
\begin{proof} Already proved (1) $\Rightarrow$ (2) $\Rightarrow$ (3) $\Rightarrow$ (4) $\Rightarrow$ (5) $\Rightarrow$ (6) $\Rightarrow$ (7).
	
	For (7) $\Rightarrow$ (1),
	let $(\mathcal{A}_k : k\in \mathbb{N})$ be a sequence of semi-open covers of $X.$ Then by Lemma \ref{PP}, for each $x \in X $,  we have a $B_{k,x} \in SO(X)$ such that $x\in B_{k,x} \subset sCl(B_{k,x}) \subset A$ for some $A \in \mathcal{A}_k,$.   For  $k\in \mathbb{N},$ put $\mathcal{B}_k = \{B_{k,x}: x\in X\}.$   Then $(\mathcal{B}_k : k\in \mathbb{N})$  is a sequence of semi-open covers of $X.$
	As $X$ is an  extremally disconnected, by \cite[Proposition 4.1]{U}, we have  $B\subset Int(Cl(B)). $   For each $B\in SO(X)$, $Cl( Int(Cl({B}))) $ is clopen in $X.$ Put $\mathcal{C}_k = \{Cl(Int(Cl({B}))) : B\in \mathcal{B}_k\}.$ Then $(\mathcal{C}_k : k\in \mathbb{N})$ is a sequence of clopen covers of $X.$ As   $X$ is    mildly Hurewicz,  there exists a sequence $(\mathcal{C}'_k : k\in \mathbb{N}),$ where for each $k,$ $\mathcal{C}'_k$ is a finite subset of $\mathcal{C}_k$  such that  for $x\in X$, $x\in \cup{\mathcal{C}'_k}$ for all but finitely many $k.$ Observe that for each subset $A$ of   $X$, $Int(Cl(A))\subset sCl(A)$  and from the  extremal disconnectedness of $X$, $sCl(A)  = Cl(A)$ for each $A\in SO(X).$  From the above construction,  for each $C' \in \mathcal{C}'_k$ we have a $A_{C'} \in \mathcal{A}_k$ such that $C'\subset A_{C'}.$  Then for  $k\in\mathbb{N}$, let $\mathcal{A}'_k = \{A_{C'} : C'\in \mathcal{C}'_k\}$. Hence  the sequence ($\mathcal{A}'_k : k\in \mathbb{N}$) witnesses  that  $X$ is semi-Hurewicz.  
\end{proof}

In the following examples, we show that the extremal disconnectedness and semi-regularity are neccessary conditions  in     Theorem \ref{XVB}.
\begin{example} Consider the real line $\mathbb{R}$ with usual topology.  Then $\mathbb{R}$   is   semi-regular mildly Hurewicz space but it is not an  extremally disconnected space .  On the other hand,   $\mathbb{R}$ is not  semi-Hurewicz  \cite{KSDS}. 
\end{example}
\begin{example} Let  $X$ be an uncountable cofinite space, that means an uncountable set $X$ with cofinite topology. Then $X$ is an extremally disconnected mildly Hurewicz space. But $X$ does not have  semi-Hurewicz property, since the  semi-open  cover $\{ X\setminus\{x\} : x\in X\}$ has no countable subcover.  
	
\end{example}   In an  extremally disconnected space, zero-dimensionality  and  semi-regularity are equivalent (\cite{MNA}, Theorem 6.4). We have the following corollary: 
\begin{corollary}\label{AZ}
	For  an  extremally disconnected, zero-dimensional  space $X,$ the following statements are equivalent:
	\begin{enumerate}
		\item   $X$ is   semi-Hurewicz;
		\item  $X$ is $\alpha$-Hurewicz;
		\item  $X$ is Hurewicz;
		\item $X$ is nearly Hurewicz;
		\item  $X$ is almost Hurewicz
		\item   $X$ is $\theta$-Hurewicz;
		\item   $X$ is mildly Hurewicz.
	\end{enumerate}
\end{corollary}
A space $X$ is called $S$-paracompact \cite{A} if for every
open cover of $X$ has a locally finite semi-open refinement.  A   $S$-paracompact Hausdorff space $X$ is semi-regular \cite{A}. Hence the properties mentioned in Theorem \ref{XVB} are also equivalent for an extremally disconnected    $S$-paracompact  Hausdorff  spaces.

It is  known   that the Stone-$\breve{C}$ech compactification of a discrete space is extremally disconnected compact Hausdorff space. Thus the class of Stone-$\breve{C}$ech compactification of
discrete spaces   contained in the class of  extremally disconnected S-paracompact Hausdorff spaces and it is turns out to be  subclass of extremally disconnected semi-regular spaces.

\begin{theorem}\label{AD}
	For a space $X,$ the following statements are equivalent:
	\begin{enumerate}
		\item $X$ has  $\theta$-Hurewicz property;
		\item $X$ satisfies $U_{fin}(\theta$-$\Omega, \theta$-$\mathcal{O})$
	\end{enumerate}
\end{theorem}
\begin{proof}
	1 $\Rightarrow$ 2. It  follows from the fact that each $\theta$-$\gamma$-cover of $X$ is a $\theta$-open cover of $X$.
	
	2 $\Rightarrow$ 1. Let  $(\mathcal{A}_k : k \in \mathbb{N})$  be a sequence of $\theta$-open covers of $X$. Let $\mathbb{N} = Y_1 \cup Y_2 \cup...\cup Y_m \cup ... $ be a partition of $\mathbb{N}$ into countably many pairwise
	disjoint infinite subsets. For each $k$,  let $\mathcal{B}_k$ contains all sets   of the form $A_{k_1} \cup A_{k_2} \cup...\cup A_{k_n}$, $k_1 \leq...\leq k_n$, $k_i \in Y_k$, $A_{k_i} \in \mathcal{A}_k$, $i\leq n$, $n\in \mathbb{N}$. Then for each $k,$ $\mathcal{B}_k$ is a $\theta$-$\omega$-cover of $X$. Applying $U_{fin}(\theta$-$\Omega, \theta$-$\mathcal{O})$ on the sequence $(\mathcal{B}_k : k \in \mathbb{N})$, there is a sequence $(\mathcal{C}_k : k \in \mathbb{N})$, where  for each $k$,
	$\mathcal{C}_k$  is a finite subset of $ \mathcal{B}_k$ such that  $ x\in X$ $x\in \cup \mathcal{C}_k$ for all but finitely many $k.$  Assume that $\mathcal{C}_k = \{C^1_k, .....C^{m_k}_k\}$, then by the above construction,  $C^i_k = A^{k_{i_1}}_k\cup.....\cup A^{k_{i_n}}_k,$ $C_k^i \in \mathcal{C}_k$. Thus for each $k,$ we have  a finite subset $\mathcal{A}'_k$ of $\mathcal{A}_k$ such that $\cup\mathcal{C}_k\subseteq \cup\mathcal{A}'_k.$ Hence  $X$ has the $\theta$-Hurewicz property.
\end{proof}
On the similar lines, we can prove that a space	$X$ has the $\alpha$-Hurewicz property if and only if
$X$ satisfies the selection principle $U_{fin}(\alpha$-$\Omega, \alpha$-$\mathcal{O})$

\begin{theorem}\label{BC} If each finite power of space $X$ is $\theta$-Hurewicz, then $X$ satisfy 	$U_{fin}(\theta$-$\Omega, \theta$-$\Omega)$.
\end{theorem} 
\begin{proof}
	Let $(\mathcal{A}_k: k \in \mathbb{N})$ be a sequence of open $\theta$-$\omega$-covers of $X$. For each $ l \in \mathbb{N}$, we put $\mathcal{B}_k =
	\{A^l: A \in \mathcal{A}_k\}$. For each $l \in \mathbb{N},$ applying the $\theta$-Hurewicz property  to the sequence $(\mathcal{B}_k: k\in \mathbb{N})$ of
	$\theta$-open covers of $X^l$,  for each $k \in \mathbb{N}$ we have finite subfamilies $\mathcal{C}_{k}$ of $\mathcal{B}_k $ such that  $x\in X^l$, $ x\in \cup C_k $ for all but finitely many $k$. For  $k \in \mathbb{N}$, let $\mathcal{A}'_k = \{A\in \mathcal{A}_k : A^l\in \mathcal{C}_k\}$. Then the sequence
	$(\mathcal{A}'_k : k\in \mathbb{N}) $ witnesses that $X$ satisfies  $U_{fin}(\theta$-$\Omega, \theta$-$\Omega)$.
\end{proof}

In a Similar way, we can prove that  if each finite power of space $X$ is $\alpha$-Hurewicz, then $X$ satisfy 	$U_{fin}(\alpha$-$\Omega, \alpha$-$\Omega)$.


\begin{lemma} \cite{MNA1}\label{ZA}
	Let $X$ be an extremally disconnected   space.   Then for  each $\theta$-$\omega$-cover  $\mathcal{A}$ of $X^k$, $k\in\mathbb{N}$, there exists  a $\theta$-$\omega$-cover $\mathcal{B}$ of $X$ such that the  $\theta$-open cover $\{B^k: B \in \mathcal{B}\}$ 
	of $X^k$ refines $\mathcal{A}$. 
\end{lemma}

\begin{theorem}\label{AG}
	Let $X$ be an extremally disconnected space. If $X$ has a property 
	$U_{fin}(\theta$-$\Omega, \theta$-$\Omega)$, then for each $n\in \mathbb{N}$, $X^n$ also has this property. 
\end{theorem}
\begin{proof}
	Let $(\mathcal{A}_k: k \in \mathbb{N})$ be a sequence of $\theta$-$\omega$-covers of $X^n$. Then by Lemma \ref{ZA},  there exists a $\theta$-$\omega$-cover $\mathcal{B}_k$ of $X$ such that $\{B^n: B \in \mathcal{B}_k\}$ refines $\mathcal{A}_k$.  Apply the condition $U_{fin}(\theta$-$\Omega, \theta$-$\Omega)$ of $X$ on the sequence $(\mathcal{B}_k : k\in \mathbb{N})$,  then for each $k\in\mathbb{N}$, there exists a finite subset $\mathcal{C}_k$ of $\mathcal{B}_k$ such that  $\{\cup\mathcal{C}_k : k\in \mathbb{N}\}$
	forms  a $\theta$-$\omega$-cover of $X$.  Since  $\{B^n: B \in \mathcal{B}_k\}$ refines $\mathcal{A}_k$,  for each $C\in\mathcal{C}_k,$ we have   $A_C\in\mathcal{A}_k$ such that  $C^n\subset A_C.$ For $k\in \mathbb{N},$ let  $\mathcal{A}'_k = \{ A_C\in \mathcal{A}_k : C\in \mathcal{C}_k\}$. Thus the sequence $(\mathcal{A}'_k : k\in \mathbb{N})$ witnesses that  $X^n$ has a property $U_{fin}(\theta$-$\Omega, \theta$-$\Omega)$. 
\end{proof}
Thus from Theorem \ref{AD},  Theorem \ref{BC} and  Theorem \ref{AG},  we obtained the following corollary. 
\begin{corollary}\label{BHO}
	Let $X$ be an extremally disconnected space. Then every finite power of  $X$ is $\theta$-Hurewicz if and only if $X$  satisfies $U_{fin}(\theta$-$\Omega, \theta$-$\Omega).$
\end{corollary}

\section{Preservation in subspaces and mappings }
In this section, we analyse  the properties of $\alpha$-Hurewicz and $\theta$-Hurewicz spaces.  We investigate the behaviour of these properties under subspaces and various type of mappings.  In the following  example we show that $\alpha$-Hurewicz is not a hereditarty property.
\begin{example}\label{MC}  Let   $x_0$ be a fixed point of an uncountable set $X.$  Then the family $\tau= \{A\subset X : x_0\notin A\}\cup \{A\subset X : X\setminus A$ is finite$\}$ of subsets of $X$ forms a topology on $X.$  It is easy to prove that  the space $(X, \tau)$ is $\alpha$-Hurewicz.  Consider the subspace $Y = X\setminus \{x_0\}$ of $(X, \tau)$. Then one point set $\{x\}, x\in Y$ is $\alpha$-open in $Y$. Then the  $\alpha$-open cover $\mathcal{A} = \{\{x\} : x\in Y\}$ of $Y$ has no countable subcover. Hence the subspace $Y$ of  the space $(X, \tau)$ is not $\alpha$-Hurewicz. 
\end{example} Note that,  $Y$  is also  an open ($\alpha$-open) subset of $(X, \tau)	$. Hence the   open ($\alpha$-open) subspace of  a $\alpha$-Hurewicz space need not be   $\alpha$-Hurewicz.

\textbf{Remark:}	Let $X$ be the   space  considered in  Example \ref{MC}  It is also easy to prove that $X$ is  $\theta$-Hurewicz space. On the other hand  the open subspace $Y = X\setminus\{x_0\}$  of $X$ is not $\theta$-Hurewicz. It means that $\theta$-Hurewicz property is  also  not  hereditary.

However the $\alpha$-Hurewicz $\&$ $\theta$-Hurewicz properties are preserved under clopen subsets as shown below in Proposition \ref{IL}. 
\begin{proposition}\label{IL}
	A clopen subspace of a $\alpha$-Hurewicz ($\theta$-Hurewicz) space is $\alpha$-Hurewicz ($\theta$-Hurewicz).
	\begin{proof}
		Let $Y$ be a clopen subspace of a $\alpha$-Hurewicz space $X$. Let $(\mathcal{A}_k :{k \in\mathbb{N}})$ be a sequence of $\alpha$-open covers of $Y$. Then $(\mathcal{B}_k :{k \in\mathbb{N}})$ is a sequence of $\alpha$-open covers of $X$, where   $\mathcal{B}_k= \mathcal{A}_k\cup\{X\setminus Y\}$ for each $k$.  Since $X$ is $\alpha$-Hurewicz, there is a sequence $(\mathcal{B}'_k : {k \in\mathbb{N}})$, where $\mathcal{B}'_k$ is a finite subset of $\mathcal{B}_k$ such that $x\in X$, $x\in\bigcup\mathcal{B}'_k$ for all but finitely many $k.$ We observe that for each $y\in Y$, $y\in\bigcup\mathcal{B}'_k\setminus\{X\setminus Y\}$. That means that $Y$ is an $\alpha$-Hurewicz space. Similarly, we can prove for $\theta$-Hurewicz space.
	\end{proof}
\end{proposition}

\begin{proposition}\label{PPP}
	Let $Y$ be a subspace of  a space $X$. If  $Y$ is $\theta$-Hurewicz, then for each sequence $(\mathcal{A}_k : k\in \mathbb{N})$ 	of covers of $Y$ by $\theta$-open sets of $X$, there is a sequence $(\mathcal{B}_k : k\in \mathbb{N})$, where for each $k$, $\mathcal{B}_k$ is a finite subset of $\mathcal{A}_k$ such that for each $y\in Y$, $y\in \cup\mathcal{B}_k$ for all except finitely many $k$.
	
\end{proposition}
\begin{proof}
	Let $Y$ be $\theta$-Hurewicz subspace of a space $X$. Let $(\mathcal{A}_k : k\in \mathbb{N})$
	be a sequence of covers of $Y$ by $\theta$-open
	sets of $X$. Put $\mathcal{B}_k = \{Y\cap A  : A\in \mathcal{A}_k\}$. Then $(\mathcal{B}_k : k\in \mathbb{N})$ is a sequence of $\theta$-open covers of $Y$ and $Y$ is  $\theta$-Hurewicz, there exists a finite subset $\mathcal{C}_k$ of $\mathcal{B}_k$ such that $y\in Y$, $y \in \cup\mathcal{C}_k$ for all but finitely many $k.$ Let  $\mathcal{A}'_k =  \{A\in \mathcal{A}_k : A\cap Y \in \mathcal{C}_k\}$. Then  the sequence $(\mathcal{A}'_k : k\in \mathbb{N})$ witnesses our requirement.
	
	In the following example we show that the converse of the  above theroem does not hold. 
\end{proof}
\begin{example}
	\indent Let   $U= \{u_{\alpha} : \alpha < \omega_{1}\}$, $V=\{v_i: i\in \omega\}$ and $W=\{\langle u_\alpha, v_i\rangle : \alpha < \omega_1, i\in\omega\}$.  Let $X=W\cup U\cup\{x'\}$, $x\not \in W\cup U.$   Topologize $X$ as follows:  for $u_\alpha\in U,$ $\alpha<\omega_1$ the  basic neighborhood takes of the form $A_{u_\alpha}(i) = \{u_\alpha\}\cup\{\langle u_\alpha, v_j\rangle : j\geq i,  i\in \omega\} $, the basic neighborhood of $x'$ takes of the form $A_{x'}(\alpha)= \{x'\}\cup\bigcup\{\langle u_\beta, v_i\rangle: \beta > \alpha, i\in \omega\},$ $\alpha < \omega_1$ and each point of $W$ is isolated.
	
	Consider the subspace  $Y = \{u_\alpha : \alpha< \omega_1\}\cup\{x'\}$ of the space $X$. Observe  that, the singleton set  $\{y\},$ $y\in Y,$  is $\theta$-open in $Y.$  Thus the family  $\{\{y\}: y\in Y\}$ is an uncountable $\theta$-open cover of $Y$, which has  no countable subcover. Hence  $Y$ is not $\theta$-Hurewicz.
	
	Next, we show that $Y$ for each sequence $(\mathcal{A}_k : k\in \mathbb{N})$ 	of covers of $Y$ by $\theta$-open sets of $X$, there exists a sequence $(\mathcal{B}_k : k\in \mathbb{N})$, where for each $k$, $\mathcal{B}_k$ is a finite subset of $\mathcal{A}_k$ such that for each $y\in Y$, $y\in \cup\mathcal{B}_k$ for all but finitely many $k$.
	
	Let $(\mathcal{A}_k : k\in \mathbb{N})$ be a family of   $\theta$-open sets  of $X$ such that for each $k\in \mathbb{N}$,  $Y\subseteq \cup\mathcal{A}_k.$ Then  for each $k\in \mathbb{N}$, there is an open set $B_k$ and   $A_k\in \mathcal{A}_k$   such that $x'\in B_k\subset \overline{B_k}\subset A_k.$  From the construction of topology on $X$  for each $k$,  there exists a $\beta_k <\omega_1$ such that $A_{x'}(\beta_k)\subseteq B_k$,  $\overline{A_{x'}(\beta_k)}\subseteq \overline{B_k}$. Thus for each $k,$
	$\{u_\alpha: \alpha>\beta_k\}\cup \{x'\}\subseteq \overline{B_k}\subset A_k$ and   $Y'_k= \bigcup_{\alpha\leq\beta_k}u_\alpha$ is countable.  Thus $ \left( \bigcup_{k\in \mathbb{N}} Y'_k \right)\cap Y$ is countable.  As similiar to Theorem \ref{BKM}, we can find a  subset $\mathcal{A}'_k$ of $\mathcal{A}_k$ such that for each $y\in \left( \bigcup_{k\in \mathbb{N}} Y'_k \right)\cap Y$, $y\in \cup\mathcal{A}'_k$ for all but finitely many $k$. For each $k\in \mathbb{N},$ let $\mathcal{A}''_k = \mathcal{A}'_k \cup \{A_k\}$. Then  $\mathcal{A}''_k$ is a finite subset of $\mathcal{A}_k$  and for each  $y\in Y$, $y\in\cup\mathcal{A}''_k $ for all but finitely many $k$.  
\end{example}
The mapping $f : X\rightarrow Y$  from a space $X$ to a space $Y$  is said to be :  

1.  $\alpha$-continuous \cite{P} ($\alpha$-irresolute \cite{O}) if the preimage of each open ($\alpha$-open) set of $Y$ is $\alpha$-open in $X.$

2.  $\alpha$-open (strongly $\alpha$-open) if the image of each $\alpha$-open set of $X$ is $\alpha$-open (open) in $Y$.

3.	  $\theta$-continuous (\cite{Q}, \cite{R}) (resp., strongly $\theta$-continuous \cite{N})
if for each $x \in X $ and each open set $B$ of $Y$ containing $f(x)$ there exists an 
open set $A$ of  $X$ containing $x$ such that $f(Cl(A)) \subset Cl(B)$ (resp., $f(Cl(A)) \subset
B)$.

\begin{theorem}
	An $\alpha$-continuous image of an  $\alpha$-Hurewicz space is  Hurewicz.
	\begin{proof} Let $X$  be an $\alpha$-Hurewicz space and $f : X \rightarrow Y$ be an $\alpha$-continuous map from $X$  onto a space $Y$. Let $(\mathcal{A}_k :{k \in\mathbb{N}})$ be a sequence of open covers of $Y$. Since $f$ is $\alpha$-continuous, then for each $k\in \mathbb{N}$ $ \{f^{-1}({A}_k) : A_k\in\mathcal{A}_k\}$, is a $\alpha$-open cover of $X$. Since $X$ is  $\alpha$-Hurewicz,  there is a sequence $(\mathcal{B}_k : {k \in\mathbb{N}})$ where for each $k,$ $\mathcal{B}_k$ is a finite subset of $\mathcal{A}_k$ such that $x\in X,$ $x\in\bigcup \{ f^{-1}(B) : B\in \mathcal{B}_k)\}$ for all but finitely many $k$. Consider   ${A_{B_k}} = f({B}_k)$, $k \in \mathbb{N}$. Then the sequence $(f(B) : B\in \mathcal{B}_k$ $\& $ $ {k\in\mathbb{N}})$  witness that  $Y$ is   Hurewicz. 
	\end{proof}
\end{theorem}
Similarly, we can prove the following theorem.
\begin{theorem} A $\alpha$-irresolute image of an  $\alpha$-Hurewicz space is   $\alpha$-Hurewicz.
	
\end{theorem}

Since each continuous map is $\alpha$-continuous, we have the following corollary:
\begin{corollary}
	A continuous image of an $\alpha$-Hurewicz space is Hurewicz.
\end{corollary}
\begin{theorem} 
	A strongly $\theta$-continuous image of a  $\theta$-Hurewicz space $X$ is  Hurewicz.
	\begin{proof} Let  $X$ be a  $\theta$-Hurewicz space and $f: X\rightarrow Y$ be a strongly $\theta$-continuous map from $X$  onto  a space $Y.$ Consider the sequence $(\mathcal{A}_k : {k \in\mathbb{N}})$  of open covers of $Y$.  Then for each $k,$ and for each $x\in X,$ $f(x) \in A_k$ for some $A_k\in \mathcal{A}_k.$  From the  strongly $\theta$-continuity of $f$, there is an open set $B_{x,k}$ containing $x$ such that $f(Cl({B_{x,k}}))\subset A_k.$  This means that $f^{-1}(A_k)$ is   $\theta$-open. Then for each $k\in \mathbb{N}$, $\{f^{-1} (A) : A \in \mathcal{A}_k\}$  is a  $\theta$-open cover of $X.$   As $X$ is  $\theta$-Hurewicz, there is a sequence $(\mathcal{C}_k : {k \in\mathbb{N}})$  where for each $k,$ $\mathcal{C}_k$ is a finite subset of $\mathcal{A}_k$  such that   $x\in X,$ $x\in \bigcup \{f^{-1}(C) : C\in \mathcal{C}_k\}$ for all but finitely many $k$. Then we have \begin{center}
			$Y = f(X) = f(\bigcup_{C\in\mathcal{C}_k}f^{-1}(C) )$
			= $\bigcup\mathcal{C}_k.$\end{center}
		Hence, $Y$ is Hurewicz.
	\end{proof}
\end{theorem}

\begin{theorem} 
	A $\theta$-continuous image of a  $\theta$-Hurewicz space $X$ is   $\theta$-Hurewicz.
	\begin{proof} Let $f: X\rightarrow Y$ be a $\theta$-continuous map from a  $\theta$-Hurewicz space $X$ onto  a space $Y.$ Let $(\mathcal{A}_k : {k \in\mathbb{N}})$ be a sequence of $\theta$-open covers of $Y$.  Then for each $k\in \mathbb{N}$, $\{f^{-1} (A) : A\in \mathcal{A}_k\}$ is a $\theta$-open cover of $X$, because $f$ is    $\theta$-continuous.  Since  $X$ is  $\theta$-Hurewicz, there is a sequence $(\mathcal{B}_k : {k \in\mathbb{N}}),$  where for each $k,$ $\mathcal{B}_k$ is a finite subset of $\mathcal{A}_k$ such that $x\in X,$ $x\in \bigcup_{B\in \mathcal{B}_k}f^{-1}(B)$ for all but finitely many $k$. For each $k,$ and for each $B_k\in\mathcal{B}_k$, we may choose $A_k \in \mathcal{A}_k$ such that $B_k = f^{-1} (A_k )$. Then we have \begin{center}
			$Y = f(X) = f(\bigcup_{B\in\mathcal{B}_k}f^{-1}(B) )$
			= $\bigcup\mathcal{B}_k.$\end{center}
		Hence, $Y$ is $\theta$-Hurewicz.
	\end{proof}
\end{theorem}
Since  continuity implies $\theta$-continuity, we have the following corollary:
\begin{corollary}
	A continuous image of a  $\theta$-Hurewicz space is   $\theta$-Hurewicz.
\end{corollary}

\begin{theorem}
	For a space $(X, \tau),$ the following statements are equivalent:
	\item (1) $(X, \tau)$ is $\alpha$-Hurewicz;
	\item (2) $(X, \tau)$ admits a strongly $\alpha$-open bijection onto a Hurewicz space $(Y, \tau')$.
	\begin{proof} 
		(1) $\Rightarrow$ (2): Let  $(X, \tau)$ be an $\alpha$-Hurewicz space, then $(X, \tau_\alpha)$ is Hurewicz. The identity map $I_X:(X, \tau)\rightarrow (X, \tau_\alpha)$ is a strongly $\alpha$-open bijection.\\
		\indent (2)$\Rightarrow (1):$ Assume that $f: (X, \tau)\rightarrow (Y,\tau')$ is  a strongly $\alpha$-open bijection from a space $(X,\tau)$ onto a Hurewicz space $(Y, \tau')$. Let $(\mathcal{A}_k :{k \in\mathbb{N}})$ be a sequence of $\alpha$-open covers of $(X, \tau).$ Then  for each $k\in \mathbb{N}$, $\{f({A}_k) : A_k\in \mathcal{A}_k\}$ is an open cover of $Y$. Since $Y$ is Hurewicz space, there exists a sequence $(\mathcal{B}_k : {k\in\mathbb{N}})$,  where for each $k$,  $\mathcal{B}_k$ is  a finite subset of $\mathcal{A}_k$ such that for each $y\in Y$,  $y\in \bigcup\{f(B) : B\in\mathcal{B}_k\}$ for all but finitely many $k.$ Hence for each  $x\in X$, $x\in\bigcup\mathcal{B}_k$ for all but finitely many $k.$ 
	\end{proof}	
\end{theorem}
\section{Characterizations of Variants of Hurewicz spaces}
Let  $\omega^\omega$ be  the set of all functions $f: \omega\rightarrow \omega$. The set $\omega^\omega$ is equipped with the product topology.
Define  a relation $\leq^{* }$ on $\omega^\omega$ as follows:  $f\leq^{*} g$ if $f(n) \leq g(n) $ for all but finitely many $n$.  Then the relation  $\leq^{*}$ on $\omega^\omega $ is  reflexive and transitive. Let $H$ be a subset of $\omega^\omega$. We say $H$ is a bounded  if $H$ has an upper bound with respect to $\leq^{*},$ otherwise $H$ is unbounded. We say $H$ is dominating if it is cofinal in $(\omega^\omega, \leq^*).$ Let $  \mathfrak{b} $ be the smallest cardinality of an unbounded subset of $\omega^\omega$ with respect  to $\leq^*.$ The cardinal $\mathfrak{b}$ is known as the bounding number.
It is not difficult to prove that $\omega_1 \leq \mathfrak{b} \leq \mathfrak{d} \leq \mathfrak{c}$ and it is known that $\omega_1 < \mathfrak{b} = \mathfrak{c},$ 
$ \omega_1 < \mathfrak{d} = \mathfrak{c}$ and $\omega_1 \leq \mathfrak{b} < \mathfrak{d} = \mathfrak{c}$ are all consistent with the axioms of ZFC  for more  details (see \cite{W}).

\begin{theorem}\label{CD1}
	Let $X\subset \omega^{\omega}$. If $X$ is  a mildly Hurewicz space, then $X$ is  bounded.
\end{theorem}
\begin{proof}
	Let us  assume that  $X$ be an unbounded subset of $\omega^\omega.$    For $f_x \in X$ and $n\in \omega$, let $A_n^{f_x} = \{h\in X : h(i) \in \{f_x(1), f_x(2),...f_x(n)\}, 1\leq i\leq n \}$. Then $A_n ^{f_x}$ is a basic open set of $X$, containing $f_x$.  Moreover, if $g\not \in A_n^{f_x} $, then there exists an $i\in\omega$ such that $1\leq i\leq n$ and $g(i)\not\in \{f_x(1), f_x(2),...f_x(n)\}.$ Then we have an open set   $B_n^g = \{h\in X : h(i)\in \omega\setminus \{f_x(1), f_x(2),...f_x(n)\}\}$ of  $X$ containing $g$ such that $B_n^g\cap A_n^{f_x} = \emptyset$. This implies that $g\not\in Cl_X(A_n^{f_x})$. Hence $Cl_X(A_n^{f_x})\subseteq A_n^{f_x}$ which implies that  $A_n^{f_x}$ is closed.  For each $n\in \omega$, put  $\mathcal{A}_n = \{A_n ^{f_x} : f_x\in X\}.$ Then $\mathcal{A}_n$ is a clopen cover of $X$ and  $(\mathcal{A}_n : n\in \omega)$  is a sequence of clopen covers of $X$. For  $n\in \omega$ and for  any finite subset  $\mathcal{B}_n$ of $\mathcal{A}_n$. Let $n_{f_x} = max \{f_x(1), f_x(2), ..... f_x(n) : f_x\in A_n^{f_x}\}$. Define a function $f :\omega \rightarrow \omega$ as follows: \begin{center}
		$f(n) = max \{n_{f_x} :      A_n^{f_x}\in \mathcal{B}_n  \} +1$.
	\end{center} From the assumption of  unboundedness of   $X$,  there exists $f'\in X $ such that $f'\not\leq^{*} f,$ that is $f(n) < f'(n)$ for infinitely many $n$. Hence for infinitely many $n,$ $f'\not\in A_n^{f_x}$,  $A_n^{f_x} \in \mathcal{B}_n.$  Thus $f'\not \in \bigcup \mathcal{B}_n$ for infinitely  many $n$. This means that $X$ does not have mildly Hurewicz property. This completes the proof.
\end{proof}
\begin{corollary}
	The   dominating subset $D$ of $\omega^\omega$ is not mildly Hurewicz.
\end{corollary}
By  Theorem \ref{CD1}, the  following corollaries follows directly:
\begin{corollary}\label{CD2}
	Let  $X$ be a $\theta$-Hurewicz subspace of $\omega^{\omega}$, then $X$ is  bounded.
\end{corollary}
\begin{corollary}
	Let $X$ be  a nearly Hurewicz subspace of $\omega^{\omega}$, then $X$ is  bounded.
\end{corollary}

\begin{corollary}
	Let $X$ be   an almost Hurewicz subspace of $\omega^{\omega}$, then $X$ is  bounded.
\end{corollary}
\begin{theorem}
	Let $X$ be a $\theta$-Hurewicz space. Then every $\theta$-continuous image of $X$ in $\omega^\omega$ is  bounded.
\end{theorem}
\begin{proof}  Let $F: X \rightarrow \omega^\omega$ be a $\theta$-continuous map from a $\theta$-Hurewicz space $X$ to $\omega^\omega$.  Then $F(X)$ is a $\theta$-Hurewicz space. Hence  $F(X)$ is  bounded.
\end{proof}
\begin{corollary}
	Every continuous image of a $\theta$-Hurewicz space $X$ in $\omega^\omega$ is  bounded.
	
\end{corollary}
\begin{corollary}
	Every continuous image of a Hurewicz space  $X$ in $\omega^\omega$ is bounded.
	
\end{corollary}
\begin{corollary}
	Every continuous image of a nearly Hurewicz space $X$ in $\omega^\omega$ is  bounded.
	
\end{corollary}

\begin{corollary}
	Every continuous image of  an almost Hurewicz space $X$ in $\omega^\omega$ is  bounded.
	
\end{corollary}

\begin{theorem}\label{AB}
	Let $X$ be a $\theta$-Lindelof space. If the cardinality of $X$ is  less than  $\mathfrak{b}$, then $X$ is  $\theta$-Hurewicz 
\end{theorem}
\begin{proof} Let $X$ be a $\theta$-Lindelof space with  $|X| < \mathfrak{b}$.  If $X$ is not a  $\theta$-Hurewicz space. Then there exists a sequence  $(\mathcal{A}_n : n\in \omega)$  of $\theta$-open covers of $X$ such that for each $n$ and for each finite subset  $\mathcal{B}_n$ of $ \mathcal{A}_n$, there exists a $x\in X$ such that $x\not \in \cup\mathcal{B}_n$ for infinitely many $n.$ Since $X$ is  $\theta$-Lindelof, assume that for each $n$, $\mathcal{A}_n = \{A_n^j : j\in \omega\}.$ For each $x\in X, $ define $f_x : \omega \rightarrow \omega $ as : $f_x(n) = min\{j : x\in A_n^j\}$. Let $D = \{f_x : x\in X\}.$  Then $D$ is an unbounded set. If $D$ is bounded. Then there exists  a $f\in \omega^\omega$ such that $f_x\leq^* f$  for all $f_x\in D$. For $n\in\omega$, put $\mathcal{B}_n = \{A_n^j : j\leq f(n)\}.$ Then for each $x\in X$, $x\in \cup\mathcal{B}_n$ for all but finitely many $n$. This leads to be a contradiction to the fact the there is a $x\in X$ such that $x\not\in \cup\mathcal{B}_n $ for infinitely many $n.$ Thus $D$ is an unbounded set. Hence  $\mathfrak{b}\leq |D|.$   Since $|X| <\mathfrak{b}$ and it is mapped surjectively to $D$. This leads to a contradiction. Hence $X$ must be a $\theta$-Hurewicz space.  
\end{proof}
\begin{theorem}
	Let $X$ be a mildly  Lindelof space. If the cardinality of $X$ is  less than  $\mathfrak{b}$, then $X$ is  mildly Hurewicz 
\end{theorem}
\begin{proof}
	The proof is on similar lines of the proof of Theorem \ref{AB}.
\end{proof}

\begin{corollary}
	Let $X$ be a subset of real line $\mathbb{R}.$ If $X$ is not $\theta$-Hurewicz, then $|X|\geq \mathfrak{b}$.
	
\end{corollary}
\begin{corollary}
	Let $X$ be a subset of real line $\mathbb{R}.$ If $X$ is not mildly Hurewicz, then $|X|\geq \mathfrak{b}$.
	
\end{corollary}

\textbf{ Remark}:	In \cite{B},  Velicho  defined the  $\theta$-closure operator which is   denoted by  Cl$_\theta$(A). For  $A\subset X$,  Cl$_\theta(A)$ = \{$x$ $\in$
$X$ : for each neighbourhood $U$ of $x$, $Cl(U) \cap A \not = \phi\}$ and Cl(A)$\subseteq$ Cl$_\theta$(A). Many papers have been published on $\theta$-closure operator
(see \cite{CL1, CL2, KD, CL3}).  Using the $\theta$-closure operator it is interesting to investigate the following  class of spaces. A space $X$ is called $\theta$-almost Hurewicz  if for each sequence $(\mathcal{A}_k : {k \in\mathbb{N}})$ of open   covers of $X$ there exists a sequence $(\mathcal{B}_k : {k \in\mathbb{N}}),$ where   $\mathcal{B}_k$ is a ﬁnite subset of $\mathcal{A}_k$  for each  $k$, such that for each $x\in X$, $x\in \bigcup \{Cl_\theta(Cl({B})) : B\in \mathcal{B}_k\}$  for all but finitely many $k.$  Observe that   every  almost Hurewicz spaces is almost $\theta$-Hurewicz. \\ \vspace{1cm}

 \textbf{Conflicts of interests}: The authors have no relevant financial or non-financial interests to disclose.

\end{document}